\documentclass[12pt]{amsart}
\usepackage[utf8]{inputenc}
\usepackage[margin=1in]{geometry}
\usepackage{indentfirst}
\usepackage{amssymb}
\usepackage{amsmath} 
\usepackage{graphicx}
\usepackage{graphics}
\usepackage{xcolor}
\usepackage{tikz}

\theoremstyle{plain}

\newtheorem{theorem}{Theorem}

\title{Minimal Posets with Prescribed Maximal Chain Cardinalities}
\author{Todd Bichoupan}
\date{May 26, 2023}

\begin{document}

\begin{abstract}
Given a nonempty finite multiset $S$ of positive integers, we wish to find a partially ordered set $P$ of minimal cardinality such that the multiset of cardinalities of all maximal chains in $P$ equals $S$. This paper establishes upper and lower bounds on the size of $P$: $\max(S) + \lceil \log_2 |S| \rceil <= |P| <= \max(S) + |S| - 1$, and both bounds are tight.
\end{abstract}

\maketitle

\section{Background}

A \emph{partial order} $\leq$ on a set $P$ is a binary relation that satisfies the following three conditions for all $x, y, z \in P$:
\begin{itemize}
    \item Reflexivity: $x \leq x$
    \item Antisymmetry: $x \leq y \land y \leq x \implies x = y$
    \item Transitivity: $x \leq y \land y \leq z \implies x \leq z$
\end{itemize}
A set equipped with a partial order is called a \emph{partially ordered set}, or \emph{poset} for short. As a convention, $x < y$ for $x, y \in P$ if $x \leq y$ and $x \neq y$. For any subset $S$ of $P$, the relation $\leq$ is still a valid partial order when it is restricted to the elements of $S$, and the poset $(S, \leq)$ is called a \emph{suborder} of $(P, \leq)$. \newline

Every poset $(P, \leq)$ has an associated covering relation, $\prec$. For $x, y \in P$, $x \prec y$ if $x < y$ and there is no $z \in P$ such that $x < z$ and $z < y$. When $P$ is finite, the partial order can be recovered from the covering relation: $x \leq y$ if there is sequence $\{a_i\}_{i=0}^n$ of members of $P$ such that $a_0 = x, a_n = y$, and $a_i \prec a_{i+1}$ for all $i < n$. When $P$ is finite, it can be visualized with a \emph{Hasse diagram}, which is drawn by arranging the elements of $P$ so that $y$ is above $x$ whenever $x < y$, and $x$ is connected to $y$ if $x \prec y$. \newline

A covering relation on a poset $(P, \leq)$ forms a simple directed graph on the members of $P$. Moreover, a simple directed graph is the covering relation associated with some partial order if and only if the graph is acyclic and every edge is the unique path connecting its endpoints (that is, if an edge connects a vertex $x$ to a vertex $y$, then there is no other path from $x$ to $y$). \newline

For a poset $(P, \leq)$, elements $x, y \in P$ are \emph{comparable} if $x \leq y$ or $y \leq x$. A set $C \subseteq P$ is a \emph{chain} if all the elements of $C$ are comparable to one another. A chain $C$ in $P$ is \emph{maximal} if no other chain in $P$ is a proper superset of $C$. \newline

For a finite poset $(P, \leq)$, chains and maximal chains can also be defined in terms of the covering relation. A nonempty finite set $C \subseteq P$ is a chain if and only if the elements of $C$ can be ordered in a sequence $\{a_i\}_{i=0}^n$ where $a_i \prec a_{i+1}$ for all $i < n$, and $C$ is maximal if and only if $a_0$ is a minimal element of $P$ and $a_n$ is a maximal element of $P$. In other words, a chain is a path in the directed graph associated with the covering relation, and a chain is maximal if the path cannot be extended by adding an element to either end.
\newline

\section{Counting Chains with Matrices} \label{CMat}

Let $(P, \leq)$ be a nonempty finite poset. Let $n = |P|$ and index the elements of $P$ with the nonnegative integers less than $n$, so $P = \{p_i\}_{i=0}^{n-1}$. The \emph{adjacency matrix} associated with the covering relation on $(P, \leq)$ is the $n$-by-$n$ matrix $A$ where $A_{ij} = 1$ if $p_i \prec p_j$ and $A_{ij} = 0$ otherwise. It is well known that one can efficiently calculate the number of paths between two vertices of a graph by taking powers of the corresponding adjacency matrix. In particular, for any nonnegative integer $k$, entry $(i, j)$ of $A^k$ is the number of chains $C \subseteq P$ such that $p_i$ is the minimum element of $C$, $p_j$ is the maximum element of $C$, and $|C| = k + 1$ ($A^0$ is the identity matrix by convention). By summing all the entries $(i, j)$ of $A^k$ where $p_i$ is a minimal element of $P$ and $p_j$ is a maximal element of $P$, one obtains the total number of maximal chains of $P$ that have cardinality $k+1$. It is worth noting that since the graph associated with $A$ is acyclic, $A^k$ is only nonzero for $k < n$. \newline

\section{Problem Statement}

Given a nonempty finite poset $(P, \leq)$, the previous section shows how one can obtain the multiset $S$ of cardinalities of maximal chains in $P$. We are interested in the opposite problem: we are given a nonempty finite multiset $S$ of positive integers, and we are looking for a poset $(P, \leq)$ such that the multiset of maximal chain cardinalities equals $S$; furthermore, we want to find a poset $(P, \leq)$ that minimizes $|P|$ while satisfying the constraint on its maximal chain cardinalities. \newline

\section{Trivial Bounds} \label{TrB}

Suppose $S$ is a nonempty finite multiset of positive integers. Let $m$ be the maximum element of $S$ and let $n$ be the cardinality of $S$ (counting multiplicities). If $(P, \leq)$ is a poset whose multiset of maximal chain cardinalities equals $S$, then $|P| \geq m$, since $P$ must contain a chain of size $m$. On the other hand, it is always possible to construct a poset $(P, \leq)$ where $|P| = m + n - 1$ and the multiset of maximal chain cardinalities equals $S$. \newline

Let $S' = S \setminus \{ m \}$ (so the multiplicity of $m$ in $S$ minus the multiplicity of $m$ in $S'$ equals one). Suppose $P$ has an element $x_i$ for each positive integer $i$ less that or equal to $m$, and suppose $P$ has an element $s_{jk}$ for each distinct element $j \in S'$ and each positive integer $k$ less that or equal to the multiplicity of $j$ in $S'$. So $|P| = m + n - 1$, and if the partial order on $P$ is defined so that $x_i < x_j$ and $x_i < s_{jk}$ when $i < j$, then the associated multiset of maximal chain cardinalities equals $S$. Figure \ref{FigUB} shows the Hasse diagram for the poset given by this construction if $S = \{ 2, 3, 3, 5, 5 \}$. \newline

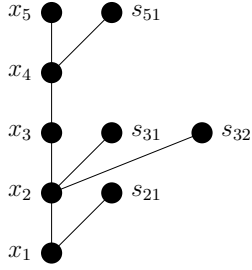
\begin{figure}[ht]
  \centering
  \scalebox{0.8}{
    \begin{tikzpicture}
      \draw (0,1) -- (0,5);

      \draw (0,1) -- (1,2);
      \draw (0,2) -- (1,3);
      \draw (0,2) -- (2.5,3);
      \draw (0,4) -- (1,5);

      \fill (0,1) circle (5pt) node[left=5pt] {$x_1$};
      \fill (0,2) circle (5pt) node[left=5pt] {$x_2$};
      \fill (0,3) circle (5pt) node[left=5pt] {$x_3$};
      \fill (0,4) circle (5pt) node[left=5pt] {$x_4$};
      \fill (0,5) circle (5pt) node[left=5pt] {$x_5$};

      \fill (1,2) circle (5pt) node[right=5pt] {$s_{21}$};
      \fill (1,3) circle (5pt) node[right=5pt] {$s_{31}$};
      \fill (2.5,3) circle (5pt) node[right=5pt] {$s_{32}$};
      \fill (1,5) circle (5pt) node[right=5pt] {$s_{51}$};
    \end{tikzpicture}
  }
  \caption{\label{FigUB}
    A poset with maximal chain cardinalities $\{ 2, 3, 3, 5, 5 \}$. \newline
  }
\end{figure}

\pagebreak

\section{The Lower Bound}

The following theorem shows that the lower bound from Section \ref{TrB} can be improved.

\begin{theorem} \label{LB}
Let $(P, \leq)$ be a finite poset and let $S$ be the multiset of all cardinalities of maximal chains in $P$. Let $m$ be the maximum element of $S$ and let $n = |S|$ (counting multiplicities). Then $|P| \geq m + \log_2(n)$.
\end{theorem}

\begin{proof}
Let $k = |P| - m$. Since $P$ must contain a maximal chain of cardinality $m$, the elements of $P$ can be partitioned into sets $M$ and $X$ such that $M$ is a maximal chain of $m$ elements and $X$ contains the remaining $k$ elements. Let $C$ be a maximal chain in $P$. Since $C \cap X$ is a subset of $X$, there are at most $2^k$ possible values of $C \cap X$. \newline

The critical idea is that we can determine which elements of $M$ are in $C$ if we only know which elements of $X$ are in $C$. By the pigeonhole principle, this will imply that there are at most $2^k$ possible maximal chains $C$ in $P$, which implies that there are at most $2^k$ different numbers in $S$. So $2^k\geq n$, or equivalently, $k \geq \log_2(n)$. \newline

Let $A = C \cap X$. If an element $y \in M$ is a member of $C$, then $y$ is comparable to every element of $A$. So if $B$ is the set of all $y \in M$ such that $y$ is comparable to all elements of $A$, then $B \supseteq C \cap M$. Furthermore, the set $A \cup B$ is a chain: every pair of elements in $A$ is comparable because $A \subseteq C$, every pair of elements in $B$ is comparable because $B \subseteq M$, and every element of $A$ is comparable to every element of $B$ by the definition of $B$. But if $A \cup B$ is a chain where $A \cup B \supseteq C$, then $A \cup B = C$ by the maximality of $C$. It follows that $C \cap M = B$.
\end{proof}

The following theorem shows that the lower bound given in Theorem \ref{LB} is sharp.
\begin{theorem} \label{Sums}
Let $k$ be a positive integer, and let $I$ be the set of all positive integers less than or equal to $k$. For each $i \in I$, let $a_i$ be a nonnegative integer. Let $A$ be the multiset $\{a_i: i \in I\}$. Let $\textup{sums}(A)$ be the multiset $\{\Sigma_{i \in J} a_i: J \subseteq I\}$ (so $|\textup{sums}(A)| = 2^k$). If $S$ is the multiset $\{k + x: x \in \textup{sums}(A)\}$, $n = |S| = 2^k$, and $m = \max(S) = k + \Sigma_{i=1}^k a_i$, then there is a poset of cardinality $m + \log_2(n) = 2k + \Sigma_{i=1}^k a_i$ whose multiset of maximal chain cardinalities equals $S$.
\end{theorem}

\begin{proof}
We give a method to construct the required poset.

In \cite[Section 1.6]{Gr11}, Gr\"atzer defines the \emph{ordinal sum} as an associative binary operation on posets: for posets $(P_1, \leq_1)$ and $(P_2, \leq_2)$, the ordinal sum $(P_1, \leq_1) + (P_2, \leq_2)$ is a poset on $P_1 \cup P_2$ that preserves the order of elements within $P_1$ and $P_2$ and places all elements of $P_1$ below all elements of $P_2$.

For each $i \in I$, let $(P_i, \leq_i)$ be a poset containing a chain of size $a_i+1$ and a single additional element not comparable to any element of the chain. If $(P, \leq) = \Sigma_{i=1}^k (P_i, \leq_i)$, then the multiset of cardinalities of maximal chains in $P$ equals $S$, and $|P| = \Sigma_{i=1}^k |P_i| = \Sigma_{i=1}^k (a_i + 2) = m + \log_2(n)$.
\end{proof}

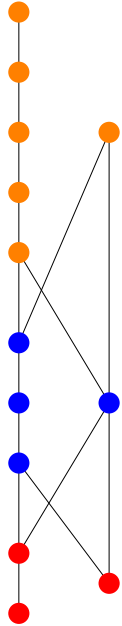
\begin{figure}[ht]
  \centering
  \scalebox{0.8}{
    \begin{tikzpicture}
      \draw (0,1) -- (0,11);
      \draw (1.5,1.5) -- (1.5,9);

      \draw (0,2) -- (1.5,4.5);
      \draw (1.5,1.5) -- (0,3.5);

      \draw (0,5.5) -- (1.5,9);
      \draw (1.5,4.5) -- (0,7);

      \fill[red] (0,1) circle (5pt);
      \fill[red] (0,2) circle (5pt);
      \fill[red] (1.5,1.5) circle (5pt);
      
      \fill[blue] (0,3.5) circle (5pt);
      \fill[blue] (0,4.5) circle (5pt);
      \fill[blue] (0,5.5) circle (5pt);
      \fill[blue] (1.5,4.5) circle (5pt);

      \fill[orange] (0,7) circle (5pt);
      \fill[orange] (0,8) circle (5pt);
      \fill[orange] (0,9) circle (5pt);
      \fill[orange] (0,10) circle (5pt);
      \fill[orange] (0,11) circle (5pt);
      \fill[orange] (1.5,9) circle (5pt);
    \end{tikzpicture}
  }
  \caption{\label{FigLB}
    A minimal poset whose multiset of maximal chain cardinalities is $\{3, 4, 5, 6, 7, 8, 9, 10\} = 3 + \textup{sums}(\{1, 2, 4\})$. Each color represents one of the posets in the ordinal sum. \newline
  }
\end{figure}

\section{The Upper Bound}

The purpose of this section is to show that the upper bound from Section \ref{TrB} is sharp. \newline

An element $x$ of a poset $(P, \leq)$ is a \emph{splitting element} of $P$ if there is more than one maximal chain in the suborder of all elements of $P$ less than or equal to $x$ and there is more than one maximal chain in the suborder of all elements of $P$ greater than or equal to $x$. Equivalently, for a finite poset $(P, \leq)$, an element $x \in P$ is splitting if the graph of the covering relation associated with $(P, \leq)$ contains more than one path from a minimal element of $P$ to $x$ and contains more than one path from $x$ to a maximal element of $P$. \newline

Suppose $(P, \leq)$ is a poset and $A$ is a subset of $P$. An \emph{in-edge} of $A$ is a pair of elements $x \in A$, $y \in P \setminus A$ where $y \prec x$. An \emph{out-edge} of $A$ is a pair of elements $x \in A$, $y \in P \setminus A$ where $x \prec y$. In other words, in-edges of $A$ are edges in the graph of the covering relation associated with $(P, \leq)$ that point from an element outside of $A$ to an element in $A$, and likewise, out-edges of $A$ are edges pointing from an element in $A$ to an element outside of $A$. \newline

Suppose that $(P, \leq)$ is a poset, $m$ is a positive integer, and $C = \{c_i\}_{i=1}^m$ is a chain in $P$ where $c_i \leq c_j$ when $i \leq j$. The \emph{index} of an in-edge of $C$ with endpoint $c_i$ is $i$, and the index of an out-edge of $C$ with endpoint $c_j$ is $j$. \newline

The following theorem shows that the upper bound from Section \ref{TrB} is sharp when the given maximal chain cardinalities are sparse.
\begin{theorem} \label{UB}
Let $S$ be a nonempty finite multiset of positive integers, let $n = |S|$, and let $m = \max(S)$. Suppose that $S$ has no element with multiplicity greater than one and that for all $a, b, c \in S$ with $a < b < c$, $(m-a) \geq (m-b) + (m-c) + n-3$. Then if $(P, \leq)$ is a finite poset where the multiset of all cardinalities of maximal chains in $P$ equals $S$, $|P| \geq m + n - 1$.
\end{theorem}

\begin{proof}
If $n \leq 3$, then $m + n + 1 = m + \lceil \log_2(n) \rceil$, and the result follows from Theorem \ref{LB}. So we may assume that $n \geq 4$. \newline

Assume for the sake of contradiction that $P$ has a splitting element, $x$. Then there are distinct maximal chains of sizes $e$ and $f$ in the suborder of elements below $x$, and there are distinct maximal chains of sizes $g$ and $h$ in the suborder of elements above $x$. So $P$ has maximal chains of sizes $e+g-1$, $e+h-1$, $f+g-1$, and $f+h-1$; since $(e+g-1) + (f+h-1) = (e+h-1) + (f+g-1)$, this implies that $S$ contains four distinct elements $a$, $b$, $c$, and $d$ with $a+d=b+c$. Assume without loss of generality that $a$ is smaller than $b$, $c$, and $d$. Then $(m-a) = (m-b) + (m-c) - (m-d) < (m-b) + (m-c) + 1$, which contradicts the condition of the theorem when $n \geq 4$. \newline

Let $M$ be the maximal chain in $P$ of size $M$, and let $X = P \setminus M$. Suppose that $M = \{t_i\}_{i=1}^m$ where $t_i \leq t_j$ when $i \leq j$. If $(x, t_i)$ is an in-edge of $M$, $(t_j, y)$ is an out-edge of $M$, and $i \leq j$, then for every integer $k$ with $i \leq k \leq j$, $t_k$ is a splitting element of $P$. Since $P$ has no splitting elements, the index of every in-edge of $M$ must be strictly greater than the index of every out-edge of $M$. Moreover, every maximal chain in $P$ includes at most one in-edge of $M$ and at most one out-edge of $M$. \newline

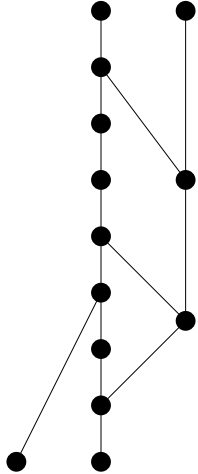
\begin{figure}[ht]
  \centering
  \scalebox{0.75}{
    \begin{tikzpicture}
      \draw (0,1) -- (0,9);
      \draw (-1.5,1) -- (0,4);
      \draw (0,2) -- (1.5,3.5);
      \draw (1.5,3.5) -- (0,5);
      \draw (1.5,3.5) -- (1.5,9);
      \draw (1.5,6) -- (0,8);

      \fill (0,1) circle (5pt);
      \fill (0,2) circle (5pt);
      \fill (0,3) circle (5pt);
      \fill (0,4) circle (5pt);
      \fill (0,5) circle (5pt);
      \fill (0,6) circle (5pt);
      \fill (0,7) circle (5pt);
      \fill (0,8) circle (5pt);
      \fill (0,9) circle (5pt);

      \fill (-1.5,1) circle (5pt);
      \fill (1.5,3.5) circle (5pt);
      \fill (1.5,6) circle (5pt);
      \fill (1.5,9) circle (5pt);
    \end{tikzpicture}
  }
  \caption{\label{FigSpl}
    A poset with no splitting elements. \newline
  }
\end{figure}

Label each maximal chain $C \neq M$ of $P$ with a pair of integers $(i, j)$, where $i$ is the index of the out-edge of $M$ whose endpoints are in $C$ or $i = 0$ if no such out-edge exists, and $j$ is the index of the in-edge of $M$ whose endpoints are in $C$ or $j = m+1$ if no such in-edge exists. By the conclusion of the last paragraph, every label $(i, j)$ has $i < j$. \newline

If $P$ beats the trivial construction and $|P| \leq m + n - 2$, then the maximal chain labels are unique. Suppose for the sake of contradiction that $P$ has two distinct maximal chains $C_1$ and $C_2$ that are not equal to $M$ and are both labeled $(i, j)$. Let $a = |C_1|$, let $b = |C_2|$, and assume without loss of generality that $a < b$. By the condition of the theorem, $(m-a) \geq (m-b) + (m-m) + n-3$, so $b-a \geq n-3$. Since $C_1 \cap M = C_2 \cap M$, $b - a = |C_2 \cap X| - |C_1 \cap X|$. $C_1$ is a maximal chain not equal to $M$, so $C_1 \not \subseteq M$ and $|C_1 \cap X| \geq 1$. Then since $|X| \leq n-2$, we must have $|C_2 \cap X| = |X| = n-2$ and $|C_1 \cap X| = 1$ to satisfy the inequality $|C_2 \cap X| - |C_1 \cap X| \geq n-3$. But then $X \subseteq C_2$, which implies that $C_1 \subsetneq C_2$, contradicting the maximality of $C_1$. \newline

Let $L_O$ be the set of all labels $(i, j)$ of maximal chains in $P$ such that $(a, c)$ is not the label of any maximal chain in $P$ for any $c < b$. Let $L_I$ be the set of all labels $(i, j)$ of maximal chains in $P$ such that $(d, b)$ is not the label of any maximal chain in $P$ for any $d > a$. Define the function $f: L_O \mapsto X$ so that for each label $(i, j) \in L_O$, $f(i,j)$ is the unique $x \in X$ such that $x$ is in the maximal chain labeled $(i, j)$ and $t_a \prec x$. Similarly, define the function $g: L_I \mapsto X$ so that for each label $(i, j) \in L_I$, $f(i,j)$ is the unique $x \in X$ such that $x$ is in the maximal chain labeled $(i, j)$ and $x \prec t_b$. $f$ and $g$ are injections, since no two labels in $L_O$ can have the same out-edge index and no two labels in $L_I$ can have the same in-edge index. So $|L_O| \leq |X|$ and $|L_I| \leq X$. \newline

Let $L = L_O \cup L_I$. Define $h: L \mapsto X$ so that for each label $(i, j) \in L$, $h(i,j)=f(i,j)$ if $(i,j) \in L_O$ and $h(i,j)=g(i,j)$ if $(i,j) \in L_I \setminus L_O$. $h$ is an injection. Assume for the sake of contradiction that $h$ is not an injection, and let $(a, b) \in L_O$ and $(c, d) \in L_I$ be labels such that $f(a, b) = h(a, b) = h(c, d) = g(c, d)$. Let $x = f(a, b) = g(c, d)$, so $t_a \prec x$ and $x \prec t_d$. Since $x < t_b$, it follows that $t_d \leq t_b$, and since $t_c < x$, it follows that $t_c \leq t_a$. But one can obtain a maximal chain labeled $(a, d)$ by taking the union of all elements less than or equal to $x$ in the maximal chain labeled $(a, b)$ with the set of all elements of $M$ above $x$. If $(a, d)$ is the label of a maximal chain in $P$ and $(a, b)$ is a label in $L_O$ with $d \leq b$, then by the definition of $L_O$, $b$ must equal $d$. Similarly, if $(a, d)$ is the label of a maximal chain in $P$ and $(c, d)$ is a label in $L_I$ with $c \leq a$, then by the definition of $L_I$, $c$ must equal $a$. Then $(a, b) = (c, d)$, contradicting the assumption that $(c, d) \not \in L_O$. So $|L| \leq |X|.$ \newline

The next paragraph will show that if $P$ beats the trivial construction and $|X| \leq n-2$, then every label of a maximal chain in $P$ other than $M$ appears in $L$. Since the label for each maximal chain is unique, it will follow that there are at most $|L| \leq |X| \leq n-2$ maximal chains in $P$ other than $M$, which shows that it is impossible for $P$ to have one maximal chain for each element of $S$. \newline

Assume for the sake of contradiction that there exists a maximal chain labeled $(a, d)$ in $P$ where $(a, d) \not \in L$. Then there is a label $(a, c) \in L_O$ and a label $(b, d) \in L_I$ where $a < b < c < d$ ($b < c$ because every out-edge of $M$ has a smaller index than every in-edge of $M$). Let $C_1$ be the maximal chain labeled $(a, c)$, let $C_2$ be the maximal chain labeled $(b, d)$, and let $C_3$ be the maximal chain labeled $(a, d)$. Let $q = |C_1|$, let $r = |C_2|$, and let $s = |C_3|$. Let $x = |X \cap C_1 \setminus C_2|$, let $y = |X \cap C_1 \cap C_2|$, and let $z = |X \cap C_2 \setminus C_1|$. Then $m-q=(c-a-1)-(x+y)$ and $m-r=(d-b-1)-(y+z)$, so $(m-q)+(m-r)=(c-a-1)-(x+y)+(d-b-1)-(y+z)=(d-a)-(x+y+z)+(c-b-y-2)$. If $y > 0$, then $C_1 \cap C_2 \cap X$ is a chain of $y$ elements in $X$ strictly between $t_b$ and $t_c$, and because $M$ is the longest maximal chain in $P$, it follows that $c-b-1 \geq y+1$; if $y=0$, then at least $c-b-1 \geq 0 = y$. So $(m-q)+(m-r) \geq (d-a)-(x+y+z)-1$. Let $w = |X \cap C_3 \setminus (C_1 \cup C_2)|$ and let $v = |X \cap C_3 \cap (C_1 \cup C_2)|$. $(m-s)=(d-a-1)-(w+v)=(d-a)-w-v-1$, so $(m-s)-(m-q)-(m-r) \leq (d-a)-w-v-1-(d-a)+(x+y+z)+1=(x+y+z+w)-2w-v \leq |X|-2w-v$. By condition of the theorem, however, $(m-s)-(m-q)-(m-r) \geq n-3$, so $|X| \geq 2w+v+n-3$. But for $P$ to beat the trivial construction, $|X|$ cannot be any larger than $n-2$, so $2w+v \leq 1$. However, $|C_3 \cap X| \geq 1$, so $w+v \geq 1$. But it is not possible to let $w=0$ and $v=1$, because no element of $X \cap (C_1 \cup C_2)$ can simultaneously cover $t_a$ and be covered by $t_d$.
\end{proof}

\section{Validation}
We wish to show that the problem of determining whether there exists a poset of cardinality at most $t$ with a prescribed multiset $S$ of maximal chain cardinalities is in the complexity class NP. \newline

Let $S$ is a nonempty finite multiset of positive integers. Let $\textup{size}(S)$ be the sum of all the digits of each number in $S$ written in binary, counting multiplicities. That is, $\textup{size}(S) = \Sigma_{x \in S} (\lfloor \log_2(x) \rfloor + 1)$. Let $n = |S|$ and let $m = \max(S)$. $n + \lfloor \log_2(m) \rfloor \leq \textup{size}(S) \leq n (\lfloor \log_2(m) \rfloor + 1)$. \newline

Let $t$ be a positive integer less than or equal to $m + n - 1$. Suppose that there exists a poset of cardinality at most $t$ whose multiset of maximal chain cardinalities equals $S$. The goal is to show that there is a method to verify that such a poset exists in an amount of time bounded by a polynomial in $\textup{size}(S)$. \newline

Naively, one might simply construct the poset $(P, \leq)$ satisfying the necessary conditions and verify that $(P, \leq)$ gives the desired set of maximal chain cardinalities using the method described in Section \ref{CMat}. The problem is that $|P|$ can be exponentially larger than $\textup{size}(S)$. But this approach can be modified to fit our needs. \newline

Let $M$ be a maximal chain in $P$ of size $m$ and let $X = P \setminus M$. For each $x \in X$, there is at most one $y \in M$ such that $y \prec x$, and there is at most one $z \in M$ such that $x \prec z$. Let $M'$ be the subset of $M$ that contains the minimum and maximum elements of $M$ as well as every element of $M$ that either covers some element of $X$ or is covered by some element of $X$. So $|M'| \leq 2n$. Let $P' = M' \cup X$. Construct an acyclic simple directed graph $G$ on the elements of $P'$ and assign a weight to each of the edges of $G$ as follows: for any $x,y \in P'$ where $x \prec y$, add an edge from $x$ to $y$ with weight 1, and for any $x, y \in M'$ where $x \not \prec y$, add an edge from $x$ to $y$ with weight equal to $1+|\{z \in M: x < z \land z < y\}|$. Then for every maximal chain $C$ in $P$, there is a corresponding path in $G$ that starts from a minimal element and ends at maximal element where the sum of all weights of edges in the path in $G$ equals the length of $C$ (the \emph{length} of a chain $C$ is $|C|-1$). \newline

The number of vertices in $G$ is no greater than $3n-1$, and the weight of each edge is no greater than $m$. Suppose that the vertices of $G$ (i.e. the elements of $P'$) are indexed and labeled $\{ v_i \}_{i=0}^{|P'|-1}$. Let $A$ be the adjacency matrix for $G$, but instead of placing a $1$ on entries corresponding to edges of $G$, place the monomial $x^w$ on entry $(i, j)$ when there is an edge of weight $w$ connecting vertex $v_i$ to vertex $v_j$, and place $0$ on entry $(i, j)$ when there is not an edge connecting vertex $v_i$ to vertex $v_j$. Let $L = \Sigma_{k=1}^{3n-1} A^k$. Entry $(i, j)$ of $L$ is a polynomial in $x$ where the coefficient of the term $x^l$ is the total number of paths from vertex $v_i$ to vertex $v_j$ where the sum of the weights of all included edges equals $l$. By summing all the polynomials on entries $(i, j)$ of $L$ where the vertex $v_i$ of $G$ is a minimal element of $P'$ and the vertex $v_j$ of $G$ is a maximal element of $P'$, one obtains a polynomial where the coefficient of the term $x^l$ is equal to the total number of maximal chains in $P$ of length $l$ (i.e. of cardinality $l+1$). \newline

\section{Future Problems}

1. Is there an efficient algorithm that takes a nonempty finite multiset $S$ of positive integers and determines the minimum possible cardinality of a poset whose multiset of maximal chain cardinalities equals $S$, or is the problem NP-complete? \newline

2. How do the upper and lower bounds change if we change the minimization criteria on the poset? In particular, what happens if, instead of aiming to minimize the number of elements in the poset, we aim to minimize the number of edges in the graph of the associated covering relation?  \newline

\section*{Acknowledgements}

Special thanks to David Wang for his contribution in the discussion of this problem, and special thanks to Kira Adaricheva for introducing the problem and providing support in the process of writing this paper. \newline

\end{document}